\documentclass[a4paper,leqno,12pt,twoside]{amsart}
\usepackage{amsmath,amsfonts,amssymb,amsthm,amscd}
\usepackage[utf8]{inputenc}
\usepackage[english]{babel}
\usepackage[colorlinks=true,citecolor=blue, urlcolor=blue, linkcolor=blue,pagebackref]{hyperref}
\usepackage[top=1.2in,bottom=1.2in,left=1.in,right=1.in]{geometry} 
\usepackage{graphicx}
\usepackage{paralist}
\usepackage{tabto}
\usepackage{standalone}
\usepackage{tikz}
\usetikzlibrary{matrix}
\usetikzlibrary{arrows}
\usepackage{xfrac}

\usepackage[alphabetic,backrefs,lite]{amsrefs}
\usepackage[all]{xy}

\usepackage{enumerate}
\usepackage{xcolor}
\usepackage{aliascnt}

\theoremstyle{plain}
\newtheorem{thm}{Theorem}[section]

\newaliascnt{corIntr}{thmIntr}

\newaliascnt{lem}{thm}
\newtheorem{lem}[lem]{Lemma}
\aliascntresetthe{lem}

\newaliascnt{cor}{thm}

\aliascntresetthe{cor}

\newaliascnt{prop}{thm}
\newtheorem{prop}[prop]{Proposition}
\aliascntresetthe{prop}

\newaliascnt{quest}{thm}

\aliascntresetthe{quest}

\theoremstyle{definition}

\newaliascnt{rem}{thm}
\newtheorem{rem}[rem]{Remark}
\aliascntresetthe{rem}

\newaliascnt{defn}{thm}

\aliascntresetthe{defn}

\newaliascnt{ex}{thm}

\aliascntresetthe{ex}

\numberwithin{equation}{section}

\def\cA{\ensuremath{\mathcal{A}}}

\newcommand{\lra}{\longrightarrow}

\makeatletter
\newcommand{\mylabel}[2]{#2\def\@currentlabel{#2}\label{#1}}
\makeatother

\newcommand\restr[2]{{
 \left.\kern-\nulldelimiterspace 
 #1 
 \vphantom{\big|} 
 \right|_{#2} 
 }}

\title[Local geometry of (2,2)-3folds]{On the local geometry of the moduli space of $(2,2)$-threefolds in ${\mathcal A}_9$}

\author{E. Colombo} 

\address{Elisabetta Colombo \\ Universit\`a degli Studi di Pavia  \\ Dipartimento di Matematica \\ Via Cesare Saldini 50  \\   20133 Milano, Italy }
\email{elisabetta.colombo@unimi.it}

\author{P. Frediani}
\address{Paola Frediani  \\ Universit\`a degli Studi di Pavia  \\ Dipartimento di Matematica \\ Via Ferrata 1  \\ 27100 Pavia, Italy  }
 \email{paola.frediani@unipv.it}

\author{J.C. Naranjo$^{1,2}$}
 \address{Juan Carlos Naranjo \newline 1. Departament de Matem\`atiques i Inform\`atica,
Universitat de Barcelona, Gran Via de les Corts Catalanes, 585, 08007 Barcelona, Spain \newline 2. Centre de Recerca Matemàtica, Edifici C, Campus Bellaterra, 08193 Bellaterra, Spain }
 \email{jcnaranjo@ub.edu}

 \author{G.P. Pirola}
 \address{Gian Pietro Pirola  \\ Universit\`a degli Studi di Pavia  \\ Dipartimento di Matematica \\ Via Ferrata 1  \\ 27100 Pavia, Italy  }
 \email{gianpietro.pirola@unipv.it}

\date{}

\begin{document}

\begin{abstract}
We study  the local geometry of the moduli space of intermediate Jacobians of $(2,2)$-threefolds in ${\mathbb P}^2 \times {\mathbb P}^2$. 
More precisely, we prove that a composition  of the second fundamental form of the Siegel metric in $\mathcal A_9$ restricted to this moduli space, with a natural multiplication map is a nonzero holomorphic section of a vector bundle. We also describe its kernel. We use the two conic bundle structures of these threefolds, Prym theory, gaussian maps and Jacobian ideals.
\end{abstract}

\thanks{E. Colombo, P. Frediani and G.P. Pirola are members of GNSAGA (INdAM) and are partially supported by PRIN project Moduli spaces and Lie theory (2017) and by PRIN project Moduli spaces and special varieties (2022).  J.C. Naranjo is partially supported by the Spanish MINECO research project PID2019-104047GB-I00 and AGAUR project 2021 SGR 00697.}

\maketitle
\setcounter{tocdepth}{1}

\section{Introduction}
This paper deals with the local geometry of the moduli space of intermediate Jacobians of $(2,2)$-threefolds in ${\mathbb P}^2 \times {\mathbb P}^2$. These threefolds have the remarkable property of having two conic bundle structures. As pointed out by Verra in \cite{verra}, this translates into the fact that the intermediate Jacobian has two interpretations as the Prym variety of two admissible double covers of two plane sextics. In fact, using this, he proves that the restriction of the Prym map to the locus of double covers of plane sextics has degree two, thus giving a counterexample to the tetragonal  conjecture of Donagi (\cite{donagi}). 
He also proves a generic Torelli theorem for the period map of these threefolds in ${\mathcal A}_9$. 

Our aim is to investigate the local geometry of this  period map studying properties of  its second fundamental form. In fact, we consider ${\mathcal A}_9$ endowed with the orbifold K\"ahler metric induced by the symmetric metric on the Siegel space and we study the second fundamental form $II$ of the locus ${\mathcal Q}$ of intermediate Jacobians of these threefolds in  ${\mathcal A}_9$ with respect to this metric. 

One of the main difficulties is that  $II$ is non-holomorphic. However, the composition of this map with a convenient multiplication map turns out to be  holomorphic and can be described. This has been done in the case of the Jacobian and Prym loci in ${\mathcal A}_g$ (see \cite{cpt}, \cite{cfg}, \cite{cf_prym}), where this composition is given by the second gaussian map of either the canonical bundle, or the Prym-canonical one. These gaussian maps are generically of maximal rank. In contrast, an analogous composition of the second fundamental form of the moduli space of intermediate Jacobians of cubic threefolds in ${\mathcal A}_5$  with a natural multiplication map is shown to be zero in  \cite{cfnp_cubic}. 
In the case of cubic threefolds, an important tool was the relation between the Jacobian ring of the cubic threefold and that of the family of plane quintics provided by the Prym construction.

This suggested to address a similar problem to other families of intermediate Jacobians of Fano threefolds and ask  if a similar behaviour occurs, in particular for the ones that can be described in terms of Prym varieties of coverings of plane curves.

In this paper, we prove that the composition of the second fundamental form of the moduli space of intermediate Jacobians of $(2,2)$ threefolds in ${\mathcal A}_9$  with a natural multiplication map is  nonzero. More precisely, it gives a nonzero holomorphic section ${\mathcal S}$ of the normal bundle $N_{\mathcal Q /\mathcal A_9}$. 
Again, Jacobian rings of both, the threefold and the associated plane sextics, play a fundamental role. 

More precisely, for a $(2,2)$ threefold $T$ in ${\mathbb P}^2 \times {\mathbb P}^2$ there is a natural bigraded Jacobian ring $\oplus R_T^{a,b}$ such that the differential of the period map can be expressed in terms of this ring (see \cite{green}, \cite{green_LNM}). In particular, the dual of the the differential of the period map corresponds to the multiplication map
\[
 \nu: Sym^2 R^{1,1}_T \lra R^{2,2}_T.
\]
Hence, there is a natural isomorphism:
\[
N^*_{\mathcal Q /\mathcal A_9, JT}=Ker(Sym^2 R^{1,1}_T \stackrel {\nu}{\lra } R_T^{2,2}).
\]
So, the second fundamental form: 
$$II : N^*_{\mathcal Q /\mathcal A_9, JT} \lra Sym^2\Omega^1_{{\mathcal Q},JT}$$
can be seen as a map: 
$$Ker(Sym^2 R^{1,1}_T \stackrel {\nu}{\lra } R_T^{2,2}) \lra  Sym^2 R^{2,2}_T.$$

The main result of this paper is the following:
\begin{thm}\label{main_intro}
For a general $T$, the composition
$$m \circ II: N^*_{\mathcal Q /\mathcal A_9, JT}=Ker(Sym^2 R^{1,1}_T \stackrel{\nu}{\lra } R_T^{2,2}) \lra R^{4,4}_T\cong \mathbb C,$$
where $m: Sym^2 R^{2,2}_T  \lra R^{4,4}_T$ is the multiplication map, 
is nonzero and gives a canonical non-trivial section ${\mathcal S}$ of the normal bundle. 
As a consequence, the second fundamental form $II : N^*_{\mathcal Q /\mathcal A_9, JT} \lra Sym^2\Omega^1_{{\mathcal Q},JT}$ is nonzero. 
\end{thm}

Moreover, for a general $T$, the kernel of the composition $m \circ II$ is explicitly described in terms of the geometry of $T$ (see Proposition \ref{kernel}, and Remark \ref{rikernel}).

We remark that the statement that the second fundamental form is different from zero is also implied by the fact that the monodromy group of the family of these threefolds is the whole symplectic group (see Remark \ref{mono}). 

One of the main ingredients in the proof is to relate the second fundamental form $II$ with the restriction of the second fundamental form of the Prym map ${\mathcal R}_{10} \lra {\mathcal A}_9$ to the locus of double coverings of plane sextics. 
In fact, the intermediate Jacobian of $T$ is isomorphic, as principally polarized abelian variety, to the Prym variety of the two double coverings of the two plane sextics determined by the two conic bundle structures.

We show that the restriction of the composition $m \circ II$ to the space of quadrics containing the Prym-canonical image of both the plane sextics is nonzero. This is done using the fact that this restriction coincides with the composition of the second gaussian map of the Prym-canonical bundle of the plane sextics with a suitable projection.

Then we compute this composition in a specific example for the quadric containing the Prym-canonical image of a sextic that is defined by the equation of the threefold, and we show that it does not vanish. 
This implies the main theorem. 

Finally, we are also able to explicitly describe the kernel of this composition by computing the second gaussian map on a set of quadrics of rank 4, constructed using some pencils on the plane sextics. To show that this set of quadrics generically generates the kernel,  we make an explicit computation on an example using a MAGMA script. 

This allows us to completely describe the kernel of $m \circ II$, since the conormal bundle is generated by the two vector spaces of the quadrics containing the Prym-canonical images of the two plane sextics. 

The structure of the paper is as follows: 

In section 2 we recall the properties of $(2,2)$ threefolds following \cite{verra}. 
In section 3 we introduce the second fundamental form $II$, we recall the theory of bigraded Jacobian rings and we state the main Theorem. In section 4 we give an interpretation of the second fundamental form in terms of Prym theory, and we prove the main Theorem.

{\bf Acknowledgments:} We warmly thank Bert van Geemen for the fundamental  help in the computation in the example in section 4 and for the MAGMA script.

\section{$(2,2)$-threefolds in $\mathbb P^2\times \mathbb P^2$}

In this section we mainly recall the properties of $(2,2)$-threefolds proved by Verra in \cite{verra}.

We consider a threefold $T$ in  $\mathbb P^2\times \mathbb P^2$ given by a bihomogeneous equation $F$  of degree $(2,2)$. More precisely:
\[
F=\sum_{0\le i,j,k,l \le 2} a_{ijkl}x_ix_jy_ky_l.
\]
Let $W$ be the image of the Segre embedding $s:\mathbb P^2 \times \mathbb P ^2\hookrightarrow \mathbb P^8$. Then:
\[
T\subset \mathbb P^2 \times \mathbb P^2 \cong W \subset \mathbb P ^8.
\]
We denote  $\alpha _{ik}:=x_iy_k$. These are natural coordinates in $\mathbb P ^8$, and $T$ can be seen as a complete intersection $T=W\cap Q$, where $Q\subset \mathbb P^8$ is the quadric given by the equation:
\[
Q_0=\sum_{0\le i,j,k,l \le 2} a_{ijkl}\alpha_{ik}\alpha_{jl}.
\]
Remember that the quadratic equations $\alpha _{ik}\alpha _{jl}-\alpha_{il}\alpha_{jk}$ generate the ideal $I_W$ of $W$. Then the ideal of $T$ is $I_T=I_W+\langle Q_0 \rangle.$

The projections on each factor $p_i:T\lra \mathbb P^2$ provide two conic bundle structures on $T$. In other words, fixed a point $(x_0:x_1:x_2)$ in the first plane, the equation $F(x_0,x_1,x_2,y_0,y_1,y_2)$ gives a conic whose matrix $A(x)$ has degree $2$ entries in the $x_i's$, $A(x)_{kl} = \sum_{i,j} a_{ijkl} x_i x_j $. Therefore, the determinant of this matrix gives a sextic plane curve $C$ which parametrizes the degenerate conics of the family. The lines contained in these degenerate conics define a curve $\widetilde C$ contained in the corresponding Grassmannian and a natural degree two map $\pi_1:\widetilde C\mapsto C$. It was proved by Beauville in \cite{be_jac_int} that $ \pi_1$ is an admissible covering of degree $2$ and that its Prym variety $P(\widetilde C,C)$ is isomorphic (as principally polarized abelian variety) to the intermediate Jacobian $JT$. Similarly, using the second projection, there is another sextic plane curve $D$ and a covering $\pi_2:\widetilde D\mapsto D$ with the same property. In particular $P(\widetilde C,C)\cong P(\widetilde D,D)$. The main Theorem  in \cite{verra} states that the Prym map has degree exactly $2$ when restricted to the locus of unramified double coverings of plane sextics.

We will assume from now on that $T$, $C$ and $D$ are generic, in particular all three are smooth. 
We are interested in  the realization of $C$ and $D$ in $\mathbb P^8$. Let $C'$ (resp. $D'$) be the set of double points of the conics of the first (second) conic bundle structure on $T$. Notice that $C',D'\subset T \subset W\subset \mathbb P^8$ and that $C'\cong C$, $D'\cong D$. Moreover $C'$ (resp. $D'$) is the locus of points of $T$ where the partial derivatives $F_{y_0}, F_{y_1}, F_{y_2}$ (resp. $F_{x_0}, F_{x_1}, F_{x_2}$) vanish. So the corresponding ideals are:
\[
I_{C'}=I_T+\langle F_{y_0}, F_{y_1}, F_{y_2}\rangle, \qquad 
I_{D'}=I_T+\langle F_{x_0}, F_{x_1}, F_{x_2}\rangle.
\]
In fact, $C'$ is the Prym-canonical image of $C$: let $\eta \in JC$ be the $2$-torsion point that defines the covering $\pi_1$. For a generic $C$ we have that $h^0(C,\eta(2))=3$ and  the tensor product:  
\[
H^0(C,\mathcal O_C(1)) \otimes H^0(C,\eta(2)) \mapsto H^0(C,\omega_C\otimes \eta)
\]
is an isomorphism. Then the embedding $\phi: C \hookrightarrow {\mathbb P}^8$ is the composition: 
\[
C \,\stackrel{\varphi_1\times \varphi_2}{\lra }\,\mathbb P^2\times \mathbb P^2 \stackrel{\cong}{\lra }W \hookrightarrow \mathbb P^8 \cong \mathbb P(H^0(C,\omega_C \otimes \eta)^{\vee})
\]
where $\varphi_1, \varphi_2$ are the maps defined by the linear systems $|\mathcal O_C(1) |$ and $|\eta (2)|$. Thus, by definition, $C'=(\varphi_1 \times \varphi_2 )(C)$. The same holds for the cover $\pi_2$ and we denote by $\eta' \in JD$ the corresponding two-torsion point.

\section{$2$nd fundamental form}

Following \cite{verra}, we consider the following moduli spaces:
\[
\begin{aligned}
    &\mathcal Q=\{T\subset \mathbb P^2\times \mathbb P^2\mid T \text{ is a } (2,2)\text{ smooth threefold}\}/\cong , \\
    &\widetilde {\mathcal Q} =\{(T,p_i) \mid T\in \mathcal Q, \, p_i \text{ projection on }\mathbb P^2\}/\cong , \\
    &\mathcal P_6=\{\text{admissible double coverings }\tilde  C \to C, \text{ where }C \text{ is a plane sextic} \}/\cong , \\
    & \mathcal A_9= \text { moduli space of principally polarized abelian varieties of dimension }9.
\end{aligned}
\]
Then, there is a commutative diagram:
\begin{equation}
\label{primo}
\xymatrix@R=0.8cm@C=0.8cm{
 \widetilde {\mathcal Q}  \ar[r]^{d}  \ar[d]_{f}    & \mathcal P_6 \ar[d]^{p}\\ 
   \mathcal Q \ar[r]^{j}  &\mathcal A_{9},  
}
\end{equation}
where $d$ associates to the conic bundle $p_i:T \lra \mathbb P^2$ the discriminant curve $C\subset \mathbb P^2$ and the natural admissible covering $\widetilde C\lra C$; $f$ is the forgetful map; $j(T)=H^{1,2}(T)/H^3(T,\mathbb Z)$ is the intermediate Jacobian of $T$; and $p$ is the Prym map restricted to $\mathcal P_6$. Moreover $\widetilde {\mathcal Q}$, $\mathcal Q$ and $\mathcal P_6$ are irreducible of dimension $19$.

The main result in \cite{verra} is the computation of the degrees of all these maps that turn out to be  generically finite on their images. He proves, based on results in \cite{be_det}, that $d$ and $j$ have degree $1$ and that $f$ and $p$ have degree $2$.

Notice that (using that $d$ has degree $1$):
\[
\Omega ^1_{\mathcal Q,\bullet  }\cong \Omega ^1_{\widetilde {\mathcal Q},\bullet }
\cong \Omega ^1_{\mathcal P_6,\bullet }
\cong \Omega ^1_{\mathcal M_{10}^{pl},\bullet },
\]
where we denote by $\mathcal M_{10}^{pl}$ the moduli space of smooth plane sextics in $\mathcal M_{10}$, the moduli space of genus $10$ curves.

Recall that $\mathcal A_9$ has a natural orbifold K\"ahler form, which is the one induced by the natural symmetric form on the Siegel space.
Our aim is to find some properties of the second fundamental form on $\overline{j(\mathcal Q)}=\overline {p(\mathcal P_6)}$ associated with the restriction of this K\"ahler metric via the embedding in $\mathcal A_9$.
By abuse of notation, we also denote by $\mathcal Q$ the closure of its image in $\mathcal A_9$.

We have the cotangent exact sequence
\[0 \rightarrow {N_{{\mathcal Q} /\cA_9}^*} \rightarrow {\Omega^1_{\cA_9}}_{\vert {\mathcal Q}} \stackrel{\hspace{-0.2cm}^tdj}\rightarrow \Omega^1_{\mathcal Q} \rightarrow 0.
\]

Denote by $\nabla$ the Chern connection of the Siegel metric and consider the second fundamental form 
\[
II=(^tdj \otimes Id) \circ \nabla_{| {N_{{\mathcal Q} /\cA_9}^*} }:   {N_{{\mathcal Q} /\cA_9}^*} \rightarrow Sym^2 \Omega^1_{\mathcal Q}.
\]

Then, we want to study the map, in a generic point $j(T) = JT \in \mathcal Q$:
\[
N^*_{\mathcal Q/\mathcal A_9, JT}
\stackrel{II}{\lra}
Sym^2 \Omega ^1_{\mathcal Q,JT }.
\]

The conormal sheaf of $\mathcal Q$ in $\mathcal A_9$ is the kernel of the dual the differential of $j$:
\[
^tdj: \Omega ^1_{\mathcal A_9 \vert \mathcal Q,JT}\cong Sym^2 H^{2,1} (T)\lra \Omega_{\mathcal Q,JT}^1.
\]
Recall that Griffiths studied the periods of hypersurfaces by means of the Jacobian ring. Later, Green extended this theory to a more general setting, covering in particular, the case of hypersurfaces in the product of two projective spaces (see \cite{green} and  \cite[Lecture 4]{green_LNM}). More precisely,
we consider the bigraded Jacobian ring $\oplus R^{a,b}_T$ of $T$, which is the quotient of the ring of bihomogeneous polynomials $\oplus S^{a,b}$ by the bihomogeneous ideal  $\oplus J^{a,b}_T$ generated by the partial derivatives of $F$.

Using \cite{green},  $T_{\mathcal Q,T}$ can be identified  with $R^{2,2}_T$. Moreover ${\mathcal O}_{\mathcal Q, T} \cong R^{4,4}_T\cong \mathbb C$. 
Then, there are isomorphisms:
$H^{2,1}(T)\cong R^{1,1}_T$ and $ \Omega_{\mathcal Q,T}^1\cong R^{2,2}_T$, and $^tdj$ identifies with the multiplication map:
\[
 \nu: Sym^2 R^{1,1}_T \lra R^{2,2}_T.
\]
Hence  there is a natural isomorphism:
\[
N^*_{\mathcal Q /\mathcal A_9, JT}=Ker(Sym^2 R^{1,1}_T \stackrel {\nu}{\lra } R_T^{2,2}).
\]

Then, we can compose $II$ with the multiplication map: 
$$m: Sym^2R^{2,2}_T \lra R^{4,4}_T\cong \mathbb C,$$
and $m \circ II: N^*_{\mathcal Q /\mathcal A_9} \lra {\mathcal O}_{\mathcal Q}$ gives a holomorphic section ${\mathcal S}$ of the normal bundle $N_{\mathcal Q /\mathcal A_9} $.

As stated in Theorem \ref{main_intro}, the main result of this paper is the following:

\begin{thm}\label{main}
For a general $T$, the composition: 
$$m \circ II: N^*_{\mathcal Q /\mathcal A_9, JT}=Ker(Sym^2 R^{1,1}_T \stackrel{m}{\lra } R_T^{2,2}) \lra R^{4,4}_T \cong {\mathbb C}$$
is nonzero, and it gives a non-trivial canonical holomorphic section of the normal bundle $N_{\mathcal Q /\mathcal A_9} $.

In particular, the second fundamental form: 
$$II: N^*_{\mathcal Q /\mathcal A_9, JT} \lra Sym^2\Omega^1_{{\mathcal Q},JT}$$
is nonzero.

\end{thm}

\section{Prym theory and proof of the main Theorem}

The proof of  Theorem \ref{main} relies on the fact that $\mathcal Q$ can be seen as the image of the Prym map restricted to $\mathcal P_6$. Let $C$ be a smooth sextic plane curve, and let $R^{\bullet }_C$ be the Jacobian ring of $C$. So the identification $\Omega ^1_{\mathcal M_{10}^{pl},C }\cong R_C^6$ holds. 
Then, we use Prym theory to express the normal sheaf of $\mathcal Q$ in $\mathcal A_9$  in another way. Consider the restriction $p$ of the Prym map $p_{10}$ to $\mathcal P_6 $: 
\[
\mathcal P_6 \hookrightarrow \mathcal R_{10} \stackrel {p_{10}}{\lra } \mathcal A_9.
\]
Recall that the map $p$ is generically finite of degree 2.  The transpose of $dp_{10}$ at a point $(C,\eta) \in {\mathcal R}_{10}$ is the multiplication map $Sym^2H^0(C,\omega_C\otimes \eta)\lra H^0(C,\omega_C^2)$. In a generic point   $(C,\eta)\in \mathcal P_{6}$ it is surjective by \cite[3.6]{verra}.   Thus, we have: 
\[
N^*_{\mathcal R_{10} / \mathcal A_9,(C,\eta)}\cong I_2(\omega_C\otimes \eta),\]
which  is the vector space of equations of quadrics through $C'$, the Prym-canonical image of $C$. We have a short exact sequence of conormal bundles:
\[
0\lra I_2(\omega_C\otimes \eta) \lra N^*_{\mathcal Q /\mathcal A_9, T} \lra N^*_{\mathcal P_6 /\mathcal R_{10}, (C,\eta)}=J^6_C/\langle C \rangle \lra 0.
\]
Here $J^{\bullet}_C$ means the jacobian ideal of $C$, and $\langle C \rangle $ the subideal generated by the equation of the sextic $C$.


We have the following diagram:
\begin{equation}
\label{secondo}
\xymatrix@R=0.8cm@C=0.8cm{
 N^*_{\mathcal R_{10}/\mathcal A_{9},(C,\eta)} \cong I_2(\omega_C\otimes \eta)  \ar[r]^{\hspace{-0.6cm} \rho}  \ar@{^{(}->}[d]  &  Sym^2 \Omega^1_{\mathcal R_{10}, (C,\eta)} \cong Sym^2H^0(C,\omega_C^2)
 \ar[d]
\\ 
N^*_{\mathcal Q /\mathcal A_9, P(C,\eta)}                \ar[r]  &  
Sym^2\Omega^1_{\mathcal P_6,(C,\eta)} \cong Sym^2 R^6_C,
}
\end{equation}
where $\rho $ is the second fundamental form of the Prym map $p_{10}$ and the second horizontal row is another way of representing the map $II$ using  Prym theory. 
Notice that the first vertical arrow is the inclusion of the conormal bundles, while the second vertical arrow is induced by projection of cotangent bundles.  

Recall that, by \cite{cf_prym}, the composition of $\rho $ with the multiplication map $Sym^2H^0(C,\omega _C^2)\lra H^0(C,\omega_C^4)$ is the second gaussian map $\mu_2:I_2(\omega_C\otimes \eta) \lra H^0(C,\omega_C^4)$. 

Observe that $H^0(C,\omega _C^2) \cong H^0(C,\mathcal O_C(6))$ and we have a map $H^0(C,\mathcal O_C(6)) \lra R^6_C$, induced by the projection $H^0({\mathbb P}^2,\mathcal O_{{\mathbb P}^2}(6)) \lra R^6_C$, since the equation of $C$ belongs to the jacobian ideal of $C$. Analogously, we have a map $\tau: H^0(C,\mathcal O_C(12)) \lra  R^{12}_C$.

So we have the following commutative diagram:
\begin{equation}\label{terzo}
\xymatrix@R=0.8cm@C=0.8cm{
I_2(\omega_C\otimes \eta) \ar[d] \ar[r]^{\hspace{-2.5cm} \rho}
\ar@/^2pc/[rr] ^{\mu_2}
& Sym ^2H^0(C,\omega_C^2)=Sym ^2H^0(C,\mathcal O_C(6)) \ar[d] \ar[r]^{\hspace{2.5cm}m }& H^0(C,\mathcal O_C(12))  \ar[d]^{\tau}
\\
N^*_{\mathcal Q /\mathcal A_9, JT} \ar[r]^{II} & Sym^2 R^6_C \ar[r]^m & R^{12}_C=\mathbb C.
}
\end{equation}

We will prove that the composition $\tau \circ \mu_2$ is nonzero, showing on an explicit example that $\tau \circ \mu_2(F) \neq 0$, where:
\[
F=\sum_{0\le i,j,k,l \le 2} a_{ijkl}x_ix_jy_ky_l = \sum_{0\le i,j,k,l \le 2} a_{ijkl} \alpha_{ik} \alpha_{jl} \in I_2(\omega_C \otimes \eta)
\]
denotes as usual the equation of the threefold $T$. 

\begin{rem}
    \label{bau}
Observe that, for a general $T$, by symmetry, similar diagrams and maps exist for the curve $D$. 
Notice that, by  \cite[proof of Proposition 5.3]{verra} the intersection of $I_2(\omega_C\otimes \eta)$ and $I_2(\omega_D\otimes \eta')$ is the ideal of the equations of the quadrics containing the threefold $T$ which is $10$-dimensional. Therefore, the sum 
$I_2(\omega_C\otimes \eta)+I_2(\omega_D\otimes \eta')$ has dimension $26$ and hence it equals 
$N^*_{\mathcal Q /\mathcal A_9,T}$. So the map $m \circ II$ is holomorphic, since its restriction to both $I_2(\omega_C \otimes \eta)$ and $I_2(\omega_D \otimes \eta')$ is holomorphic. 
\end{rem}

Denote by $A(x)$ the symmetric matrix given by $A(x)_{kl} = \sum_{i,j} a_{ijkl} x_i x_j$, and by $\hat{A}_{ij}(x)$ the product of the determinant of the matrix obtained from $A(x)$ by removing the $i$-th row and the $j$-th column with $(-1)^{i+j}$.

\begin{lem}
We have the following commutative diagram: 
\begin{equation}
\label{yiyj}
\xymatrix@R=0.8cm@C=0.8cm{
C \ar[d]^{\varphi_1} \ar[r]^{\varphi_2} & {\mathbb P}^2 \ar[d]^{v_2 }
\\
{\mathbb P}^2 \ar[r]^{h} & {\mathbb P}^5,
}
\end{equation}
where $h(x) = (\hat{A}_{ij}(x))_{i,j}$, $v_2$ is the Veronese map, and $\varphi_1, \varphi_2$ are the maps defined at the end of $\S 2$. 
\end{lem}
\begin{proof}
Since for any $x \in C$, $\phi(x)$ is the intersection point of the two lines corresponding to the singular conic $A(x)$, $\varphi_2(x) = (y_0:y_1:y_2)$ is the point in ${\mathbb P}^2$ given by $ker(A(x))$. So there exist nonzero constants $\lambda_i$ such that 
$ \lambda_i (y_0,y_1,y_2) = ((\hat{A}_{i0}(x)), (\hat{A}_{i1}(x)), (\hat{A}_{i2}(x)))$, for all $i =1,2,3$. By the symmetry of $A(x)$ we can identify $\lambda_i =y_i$ for all $i =0,1,2$. 
So the diagram is commutative and in $C$ we have: 
\begin{equation}
y_iy_j = \hat{A}_{ij}(x).
\end{equation} 
\end{proof}

Now choose a local coordinate $z$ on $C$, a local frame $l$ for ${\mathcal O}_C(1)$ and a local frame $\sigma$ for $\eta$. Then locally we write $x_i= f_i(z) l$, $y_i= g_i(z) l^2 \sigma$ and $\alpha_{ij} = f_i(z) g_j(z) l^3 \sigma$.  So the local expression for $\mu_2(F)$ is given by:
$$\mu_2(F) = -\sum_{ i,j,k,l } a_{ijkl} (f_ig_k)'' (f_j g_l) (dz)^2 l^6 =  \sum_{ i,j,k,l } a_{ijkl} (f_ig_k)' (f_j g_l)' (dz)^2l^6.$$

First, notice that since $F \in I_2(\omega_C \otimes \eta)$, we have:

$$\sum_{ i,j,k,l } a_{ijkl} (f_ig_k)(z) (f_j g_l)(z) =0,$$
hence taking the derivative with respect to $z$, we get:
$$
\begin{aligned}
0 = & \sum_{ i,j,k,l } a_{ijkl} (f_ig_k)' (f_j g_l) = \sum_{ i,j,k,l } a_{ijkl} (f'_ig_k + f_i g'_k) (f_j g_l)  \\
= & \sum_{ i,j,k,l } a_{ijkl} f'_i f_jg_kg_l + \sum_{ i,j,k,l } a_{ijkl} f_if_j g_l g'_k \\
=  & \sum_{ i,j,k,l } a_{ijkl} f'_i f_jg_kg_l + \sum_{k} g'_k(\sum_{i,j,l} a_{ijkl} f_if_j g_l) \\
= & \sum_{ i,j,k,l } a_{ijkl} f'_i f_jg_kg_l + \sum_{k} g'_k( \sum_{l}A_{kl}(z) g_l(z))=
\sum_{ i,j,k,l } a_{ijkl} f'_i f_jg_kg_l,
\end{aligned}
$$

where $A_{kl}(z) = \sum_{i,j} a_{ijkl} f_i(z) f_j(z)$, and the last equality holds since the point $(y_0,y_1,y_2)^t$ is in the kernel of the matrix $A(x)$, hence $\sum_{l}A_{kl}(z) g_l(z) =0$. 

So, differentiating with respect to $z$  the polynomial equation: 
\begin{equation}
\sum_{ i,j,k,l } a_{ijkl} f'_i f_jg_kg_l =0
\end{equation}
we obtain:
\begin{equation}
\label{''}
\sum_{ i,j,k,l } a_{ijkl} f''_i f_jg_kg_l+ \sum_{ i,j,k,l } a_{ijkl} f'_i f'_jg_kg_l + 2\sum_{ i,j,k,l } a_{ijkl} f'_i f_jg'_kg_l =0.
\end{equation}

Now 
$$
\begin{aligned}
\mu_2(F) = & -\sum_{ i,j,k,l } a_{ijkl} (f_ig_k)'' (f_j g_l) (dz)^2 l^6 \\
 = & - \left(\sum_{ i,j,k,l } a_{ijkl} f''_i f_jg_kg_l+ 2 \sum_{ i,j,k,l } a_{ijkl} f'_i f_jg'_kg_l + \sum_{ i,j,k,l } a_{ijkl} f_i f_jg''_kg_l\right) (dz)^2 l^6 \\
 = & -\left(\sum_{ i,j,k,l } a_{ijkl} f''_i f_jg_kg_l+ 2 \sum_{ i,j,k,l } a_{ijkl} f'_i f_jg'_kg_l \right) (dz)^2 l^6 \\ 
 = &  \sum_{ i,j,k,l } a_{ijkl} f'_i f'_jg_kg_l  (dz)^2 l^6,
 \end{aligned}
 $$

where the last equality follows from \eqref{''} and we used that:
$$
\sum_{ i,j,k,l } a_{ijkl} f_i f_jg''_kg_l = \sum_k g''_k \sum_{l} A_{kl}(z) g_l (z)=0.
$$

So we have found the following local expression for $\mu_2(F)$: 
\begin{equation}
\label{mu2}
\mu_2(F) =  \sum_{ i,j,k,l } a_{ijkl} f'_i f'_jg_kg_l  (dz)^2 l^6.
\end{equation}

\begin{prop}\label{cinghialino}
    The map $\tau \circ \mu_2$ is nonzero.
\end{prop}

\begin{proof}
We show the proposition by exhibiting an explicit example where the map is nonzero. 

Consider the following smooth $(2,2)$-threefold $X = \{F=0\}$, where: 
\begin{equation} \label {Fermat}
\begin{aligned}
F & =  6 x_0^2y_0^2 + 2 \lambda x_2^2 y_0y_1 +  2\lambda x_1^2 y_0y_2 + 6 x_1^2y_1^2 + 2\lambda x_0^2y_1y_2 + 6 x_2^2y_2^2 = \\
&=6 \alpha_{00}^2 + 2 \lambda \alpha_{20}\alpha_{21} + 2 \lambda \alpha_{10} \alpha_{12} + 6 \alpha_{11}^2 + 2 \lambda \alpha_{01}\alpha_{02} + 6 \alpha_{22}^2,
\end{aligned}
\end{equation}
with $\lambda^3 =-108$. 
Thus, the equation of the plane sextic attached to the first projection $p_1$ is the determinant of the following symmetric matrix:
$$
A = {\small \left( \begin{array}{ccc}
6x_0^2 & \lambda x_2^2 & \lambda x_1^2\\
\lambda x_2^2 & 6x_1^2 & \lambda x_0^2 \\
\lambda x_1^2 & \lambda x_0^2 &6x_2^2  \\
\end{array} \right),}
$$
$$det (A) =-6 \lambda^2(x_0^6 + x_1^6 + x_2^6).$$

So the plane curve $C$ is the Fermat sextic with equation $G := x_0^6 + x_1^6 + x_2^6 =0$. We will show that $\tau \circ \mu_2(F) \neq 0$. Consider the affine chart $x_0 \neq 0$,  set $w_i = x_i/x_0$, $i =1,2$, $h(w_1, w_2) := G(1, w_1, w_2) = 1 + w_1^6 + w_2^6.$ Assume that $\partial h/\partial w_2 = 6w_2^5 \neq 0,$ so $z:= w_1$ is a local coordinate, hence $w_2 = f_2(z)$. In this local setting we have $z= w_1$, $l = x_0$, $x_1 = w_1 x_0 =z l$, $x_2 = w_2 x_0 = f_2(z) l$. So in the above notation we have $f_0(z) = 1$, $f_1(z) = z$, $f_2(z) =w_2(z)$, and since $h(z, f_2(z)) = 0$, we have: 
$$\partial h/\partial w_1 + f'_2(z) \partial h/\partial w_2 =0,$$
so $f'_2(z) = -h_{w_1}/h_{w_2}$ (where $h_{w_i} = \partial h/\partial w_i$), while clearly $f'_0 = 0$, $f'_1 = 1$. 
Then, by   \eqref{mu2}, we have: 
$$
\begin{aligned}
\mu_2(F) =&  \sum_{ i,j,k,l } a_{ijkl} f'_i f'_jg_kg_l  (dz)^2 l^6 \\
         =&  \sum_{ k,l } (g_kg_l) l^4  (a_{11kl} (f'_1)^2 + 2 a_{12kl} f'_1 f'_2 +  a_{22kl} (f'_2)^2 )(dz)^2 l^2 \\
         =& \sum_{ k,l } (g_kg_l) l^4  (a_{11kl}  - 2 a_{12kl} (h_{w_1}/h_{w_2}) +  a_{22kl} (h_{w_1}/h_{w_2})^2 )(dz)^2 l^2 .
\end{aligned}
 $$
Recall that by \eqref{yiyj} we have $\hat{A}_{kl}(z) = (g_kg_l) l^4$, thus: 
$$
\begin{aligned}
\mu_2(F) = & \sum_{ k,l }  \hat{A}_{kl}(z)(a_{11kl}  - 2 a_{12kl} (h_{w_1}/h_{w_2}) +  a_{22kl} (h_{w_1}/h_{w_2})^2 )(dz)^2 l^2 \\
= &  \sum_{ k,l }  \hat{A}_{kl}(z)(a_{11kl}  - 2 a_{12kl} (h_{w_1}/h_{w_2}) +  a_{22kl} (h_{w_1}/h_{w_2})^2 )(h_{w_2})^2 l^8,
\end{aligned}
$$
since, by adjunction, we can write $dz = h_{w_2}l^3$. So we have: 
$$
\begin{aligned}
\mu_2(F) = & \sum_{ k,l }  \hat{A}_{kl}(a_{11kl}  - 2 a_{12kl} (G_{x_1}/G_{x_2}) +  a_{22kl} (G_{x_1}/G_{x_2})^2 )\frac{(G_{x_2})^2}{x_0^{10}} x_0^8 \\
= &\sum_{ k,l }  \hat{A}_{kl}(a_{11kl} (G_{x_2})^2  - 2 a_{12kl} G_{x_1}G_{x_2} +  a_{22kl} (G_{x_1})^2 ) \frac{1}{x_0^2}.
\end{aligned}
$$
Now, using our equation and computing the minors of the matrix $A$, we obtain:
$$
\begin{aligned}
\mu_2(F)= &\frac{36}{x_0^2} (x_1^{10} (2 \lambda  \hat{A}_{01} + 6 \hat{A}_{22} ) + x_2^{10} (2 \lambda  \hat{A}_{02} + 6 \hat{A}_{11} )) \\
= & -18 \lambda^2 \frac{36}{x_0^2}(x_1^{10} x_2^4 + x_2^{10} x_1^4) \\
= & -18 \lambda^2 \frac{36}{x_0^2} x_1^4 x_2^4 (x_1^6 + x_2^6).
\end{aligned}
$$
Since we are on the curve $C$, we have $x_1^6 + x_2^6 = - x_0^6$, so finally we get:
  $$\mu_2(F)= 18 \lambda^2 \frac{36}{x_0^2} x_1^4 x_2^4 x_0^6 = (18 \cdot 36) \lambda^2 x_0^4 x_1^4 x_2^4,$$
  which is not in the Jacobian ideal of $C$ that is generated by $x_0^5, x_1^5, x_2^5$. 
 Notice that if $\partial h/\partial w_1 \neq 0$, then $w_2$ is a local coordinate, so in this case we have $z =w_2$, $w_1 = f_1(z)$, $f_2(z) = z$, $f_0(z) =1$, $dw_2 = - (\partial h/\partial w_1) x_0^3 $, $f'_1(z) = -h_{w_2}/h_{w_1}$ and the computation is the same. By symmetry, an analogous computation holds in the other affine charts. 
 So $\mu_2(F)$ is the restriction of the polynomial $(18 \cdot 36) \lambda^2 x_0^4 x_1^4 x_2^4 \in H^0({\mathbb P}^2, {\mathcal O}_{{\mathbb P}^2}(12))$ to $C$  and hence $\tau \circ \mu_2(F) \neq 0$. 
  \end{proof}

Recall that by \eqref{primo} we have an identification: 
\[
R_T^{2,2} \cong \Omega^1_{\mathcal Q} \cong \Omega^1_{{\mathcal P}_6} \cong R^6_C.
\]

These identifications can be also seen as induced by the pullback via the map $\phi= \varphi_1 \times \varphi_2$: 

$\phi^*: H^0({\mathbb P}^2 \times {\mathbb P}^2, {\mathcal O}_{{\mathbb P}^2 \times {\mathbb P}^2}(2,2)) \lra H^0(C, {\mathcal O}_C(6)). $
In fact, we show the following: 
\begin{lem}\label{identification}
The space 
$J^{2,2}_T$ maps to $J^6_C/\langle C \rangle $ via $\varphi^*$, therefore there is an isomorphism  
$R_T^{2,2} \cong R^6_C$ and  a commutative diagram:
\begin{equation}
\label{quarto}
\xymatrix@R=0.8cm@C=0.8cm{
Sym ^2R_T^{2,2} \ar[d]^{\cong} \ar[r]^{m}& R_T^{4,4} \cong {\mathbb C}\ar[d]^{\cong}
\\
 Sym^2 R^6_C \ar[r]^m & R^{12}_C\cong\mathbb C.
}
\end{equation}
\end{lem}
\begin{proof}
Clearly all the elements of the form $y_j F_{y_i}$ map to zero, since they are in the ideal of $C' = \phi(C)$. So we have to show that all the elements $x_j F_{x_i}$ map to  $J^6_C/\langle C\rangle$. In fact we show that  
$$ \phi^*(F_{x_i} )= {\partial}_{x_i} (det A(x)) \in J^6_C/\langle C\rangle,$$
where the coefficients of the symmetric matrix $(A(x))_{kl} = \sum_{i,j} a_{ijkl}x_ix_j$ are quadrics in the coordinates $x_i$. 


Using  \eqref{yiyj}, we have:
$$ \phi^*(F_{x_i} )=  \phi^*({\partial}_{x_i}  ( \sum_{ij} A(x)_{ij} y_iy_j) )=  {\partial}_{x_i}  ( \sum_{ij} A(x)_{ij} \hat{A}_{ij}(x))= {\partial}_{x_i} (det A(x)),$$

In the same way, one immediately sees that $\phi^*$ induces an isomorphism $R_T^{4,4} \cong R^{12}_C$ and we have a commutative diagram as \eqref{quarto}. \end{proof}

\textit{Proof of the main Theorem \ref{main}:}
Using Proposition \ref{cinghialino} and Lemma \ref{identification}
we immediately obtain  that the restriction of $m\circ II$ to $I_2(\omega_C\otimes \eta)$ is nonzero for a generic $T$. So by Remark \eqref{bau} $m \circ II$ gives a non-trivial holomorphic section ${\mathcal S}$ of the normal bundle. 
\qed

We observe that the fact that the second fundamental form is nonzero is also implied by the following: 
\begin{rem}
\label{mono}
The monodromy group of the family ${\mathcal Q}$ is the symplectic group. Hence the special Mumford-Tate group of a generic $(2,2)$-threefold is the whole symplectic group.
\end{rem}
\begin{proof}
We prove that the monodromy group of ${\mathcal P}_6$ is the full symplectic group $Sp(18, {\mathbb Z})$. 

Recall that the monodromy map of plane sextics: 
$$
\pi_1({\mathcal M}_{10}^{pl}, C) \lra Sp(H^1(C, {\mathbb Z}))
$$
is surjective (see \cite[Theoreme 4]{beauville_mono}). 
In particular, reducing coefficients modulo 2, this implies that ${\mathcal P}_6$ is irreducible. 

For any symplectic basis $\{\alpha_1,..., \alpha_g, \beta_1, ..., \beta_g\}$ of $H^1(C, {\mathbb Z})$, where $\eta = \alpha_g$ $(\text{mod \ 2})$, there is a natural symplectic basis in $H^1(P(C, \eta), {\mathbb Z})$ (see e.g. \cite[Prop. 12.4.2]{bir_lange}). 

Consider a Lefschetz pencil around a Prym semi-abelian variety attached to a covering $\tilde{C}_0 \lra C_0$ such that $C_0$ has only one node and the two preimages of the node are two nodes of $\tilde{C}_0$. Then, writing explicitly the Picard Lefschetz transformation in terms of the basis defined above, one easily checks that the associated monodromy generates the whole symplectic group. 
\end{proof}

We will now give a complete description of the kernel of $m \circ II$. 
First, we describe the kernel of $\tau \circ \mu_2: I_2(\omega_C \otimes \eta) \lra R^{12}_C$.  

Denote by $L:= {\mathcal O}_C(1)$, $M:= {\mathcal O}_C(2)\otimes \eta$, so that $L \otimes M \cong \omega_C \otimes \eta$. 
Take a point $p \in C$ and consider the two line bundles $L(-p)= {\mathcal O}_C(1) \otimes  {\mathcal O}_C(-p)$, $M(p)={\mathcal O}_C(2) \otimes \eta \otimes  {\mathcal O}_C(p)$, then clearly $L(-p) \otimes M(p) \cong \omega_C \otimes \eta.$
Since $L$ is base point free, we have $h^0(C, L(-p)) = 2$, and we fix a basis $\{s_1, s_2 \}$ of $H^0(C, L(-p))$. 
We have a map: 
$$ \bigwedge^2 H^0(M(p)) \rightarrow I_2(\omega_C \otimes \eta),$$
\begin{equation}
\label{quadriche}
(t_1 \wedge t_2) \mapsto Q(t_1,t_2):= s_1t_1 \odot s_2t_2 - s_1t_2 \odot s_2 t_1,
\end{equation}

where $\alpha \odot \beta := \alpha \otimes \beta + \beta \otimes \alpha$. 

\begin{prop}
\label{kernel}
 For a general $T$, the quadrics defined in \eqref{quadriche}   generate the kernel of $\tau \circ \mu_2$. 
\end{prop}

\begin{proof}

A direct computation (see e.g. \cite[Lemma 2.2]{cf_michigan}) shows that we have: 
$$\mu_2: I_2(\omega_C \otimes \eta) \rightarrow H^0(C, \omega_C^{\otimes 4}) \cong H^0(C, {\mathcal O}_C(12)),$$
$$\mu_2(Q(t_1, t_2)) = \mu_{1,L(-p)}(s_1 \wedge s_2) \cdot \mu_{1,M(p)}(t_1 \wedge t_2),$$ 
where: 
$$
\mu_{1, L(-p)}: \bigwedge^2 H^0(C, L(-p)) \cong \langle s_1 \wedge s_2 \rangle\rightarrow H^0(C, L^{\otimes 2}(-2p) \otimes \omega_C) \cong H^0(C, {\mathcal O}_C(5)(-2p)), 
$$ 
$$\mu_{1, M(p)}: \bigwedge^2 H^0(C, M(p)) \rightarrow H^0(C, M^{\otimes 2}(2p) \otimes \omega_C) \cong H^0(C, {\mathcal O}_C(7)(2p))$$
denote the first gaussian maps of the line bundles $L(-p)$ and $M(p)$. 

In local coordinates, if $s_i = f_i(z) l$, where $l$ is a local frame for $L(-p)$, we have:
$$\mu_{1, L(-p)}(s_1 \wedge s_2) = (f'_1f_2 -f_1 f'_2) l^2 dz.$$
Then, clearly we obtain: 
$$div(\mu_{1,L(-p)}(s_1 \wedge s_2)) = div(Pol_p(C)_{|C}) -2p,$$
where $Pol_p(C)$ denotes the polar of the plane sextic $C$ with respect to $p$.

Choose now $t_1 \in H^0(C, M(-2p)) \subset H^0(C, M(p))$, and assume that in local coordinates we have $t_i= g_i(z) \sigma$, where $\sigma$ is a local frame for $M(p)$. Then we have:
$$\mu_{1, M(p)}(t_1 \wedge t_2) = (g'_1g_2 -g_1 g'_2) \sigma^2 dz.$$
Since $ord_pg_1 = 3$, clearly:
$$
div(\mu_{1,M(p)}(t_1 \wedge t_2)) = 2p + E,
$$
where ${\mathcal O}_C(E) \cong {\mathcal O}_C(7)$. Since the restriction to $C$ gives a surjective map: 
$$
H^0({\mathbb P}^2, {\mathcal O}_{{\mathbb P}^2}(7)) \rightarrow H^0(C, {\mathcal O}_C(7)),
$$
there exists a polynomial $G \in H^0({\mathbb P}^2, {\mathcal O}_{{\mathbb P}^2}(7))$, such that $div(G_{|C}) = E$. 
Then: 
$$
div(\mu_2(Q(t_1, t_2)) )= div(Pol_p(C)_{|C}) -2p + 2p + E = div ((Pol_p(C) \cdot G)_{|C}),
$$
hence $\tau \circ \mu_2(Q(t_1,t_2)) = 0$, since $Pol_p(C) \cdot G \in J^{12}_C$. 

It remains to show that varying $p \in C$, and choosing $t_1 \in H^0(C, M(-2p)) \subset H^0(C, M(p))$ (which is 1 dimensional for $p$ general in $C$), and varying $t_2 \in H^0(C, M(p))$, the quadrics $\{Q(t_1,t_2)\}$ generate a 17-dimensional subspace of  $I_2(\omega_C \otimes \eta).$

We will show it in the example \eqref{Fermat}, and this is enough, since it is an open property. 
To do the computation, first observe that the choice of $s_1$ corresponds to the choice of a line $l_1$ thought $p$, hence we have: 
$$div(s_1)= p + D_5,$$
where $D_5$ is an effective divisor of degree $5$. The same holds for $s_2$, so: $$div(s_2) = p + D'_5.$$
Thus, by our choice of $t_1$, 
$s_1t_1 $ is a generator of the one dimensional vector space $H^0(C, \omega_C \otimes \eta(-D_5 -3p)), $ and $s_2t_1 $ is a generator of $H^0(C, \omega_C \otimes \eta(-D'_5 -3p)). $
To choose the forms $s_1 t_2$ and $s_2 t_2$, namely to choose $t_2 \in H^0(C, M(p))$ general, it is equivalent to choose three general points $q_1, q_2, q_3 \in C$ so that: 
$$
\langle s_1t_2 \rangle = H^0(C, \omega_C \otimes \eta(-D_5 -q_1-q_2-q_3)), 
$$
$$
\langle s_2t_2 \rangle = H^0(C, \omega_C \otimes \eta(-D'_5 -q_1-q_2-q_3)).  
$$

By a MAGMA script one can check that varying a point $p$ on $C$, choosing two lines through $p$, and three points $q_1, q_2, q_3$ on $C$, one obtains a linear space of quadrics constructed as above of dimension 17.
This shows that for a generic $(C, \eta) $ the quadrics we considered generate the kernel of $\tau \circ \mu_2$. 
\end{proof}

\begin{rem}
    \label{rikernel}
By Remark \eqref{bau} we know that 
$I_2(\omega_C\otimes \eta)+I_2(\omega_D\otimes \eta')
=N^*_{\mathcal Q /\mathcal A_9,T}$. 

Hence the kernel of $m \circ II $ is generated by the quadrics of $I_2(\omega_C\otimes \eta)$ and $I_2(\omega_D\otimes \eta)$ that we have just described. 
\end{rem}


\end{document}